\documentclass[draft]{amsart}
\usepackage{amssymb,graphicx,enumerate,amsthm}
\usepackage{cite}
\usepackage{xcolor}

\renewcommand{\L}{\mathcal{L}}
\newcommand{\R}{\mathcal{R}}
\newcommand{\F}{\mathcal{F}}
\newcommand{\N}{\mathcal{N}}

\newcommand{\Z}{\mathbb{Z}}

\numberwithin{equation}{section}
\newtheorem{theorem}{Theorem}[section]

\newtheorem{corollary}[theorem]{Corollary}

\begin{document}

\makeatletter
\def\imod#1{\allowbreak\mkern10mu({\operator@font mod}\,\,#1)}
\makeatother

\author{Ali Kemal Uncu}
   \address{Research Institute for Symbolic Computation, Johannes Kepler University, Linz. Altenbergerstrasse 69
A-4040 Linz, Austria}
   \email{akuncu@risc.jku.at}

\thanks{Research of the author is supported by the Austrian Science Fund FWF, SFB50-07 and SFB50-09 Projects.}

\title[\scalebox{.9}{A Polynomial Identity Implying Schur's Partition Theorem}]{A Polynomial Identity Implying Schur's Partition Theorem}
     
\begin{abstract} 
We propose and prove a new polynomial identity that implies Schur's partition theorem. We give combinatorial interpretations of some of our expressions in the spirit of Kur\c{s}ung\"oz. We also present some related polynomial and $q$-series identities.
\end{abstract}

  

\date{\today}
   
   
\maketitle

\section{Introduction and background}\label{Sec_intro}

Since the Combinatory Analysis conference in honor of G. E. Andrews' birthday, in a series of papers Kur\c{s}ung\"oz presented his technique of writing generating functions for the number the partition functions with gap conditions on some classical partition theorems \cite{Kagan1, Kagan2, Kagan3}. His approach is backed with a combinatorial construction. This construction can be used to find finite analogs of these generating functions. Berkovich and the author \cite{BerkovichUncu7} have found finite analogs of the Capparelli's partition theorem related generating functions presented by Kanade--Russell and Kur\c{s}ung\"oz \cite{Kanade_Russell,Kagan1}. Comparing these polynomials with the earlier found finite analogs of Alladi--Andrews--Gordon and Berkovich and the author's \cite{AAG, BerkovichUncu1}, they listed polynomial identities that directly imply Capparelli's partition theorems \cite{BerkovichUncu7}. These polynomial identities led to many $q$-series relations involving the $q$-trinomial coefficients and, with the use of trinomial version of the Bailey lemma, proven infinite families of $q$-series identities in the spirit of the Andrews--Gordon Identities \cite{BerkovichUncu8, BerkovichUncu9}. Following the footsteps of \cite{BerkovichUncu7} and using other combinatorial arguments, the author presented other polynomial and $q$-series identities that are related with the classical partition theorems: namely the Euler, the Rogers--Ramanujan, the G\"ollnitz--Gordon, and the little G\"ollnitz theorems \cite{Uncu2}. It should be noted that Kur\c{s}ung\"oz also approached the G\"ollnitz--Gordon theorem \cite{Kagan1}, and the comparison of his construction versus the author equivalent formulas are discussed in \cite{Uncu2}. 

In this work, we will follow the footsteps of \cite{Kagan3, BerkovichUncu7, BerkovichUncu8, Uncu2} and present a new polynomial identity that directly implies Schur's partition theorem followed up with the study of some related $q$-series identities.

We define a \textit{partition} $\pi=(\lambda_1,\lambda_2,\dots)$ as a \textit{non-decreasing finite sequence of positive integers}, which are called parts of the partition $\pi$. We will use $\nu(\pi)$ and $|\pi|$ to denote the number of parts and the sum of all parts (size) of the partition $\pi$, respectively. The empty sequence $\emptyset$ is the only conventional partition with 0 parts and 0 size.

We start with an equivalent formulation of the Schur's partition theorem \cite{Schur}:

\begin{theorem}[Schur, 1926]\label{Schur_THM} For any non-negative integer $n$, the number of partitions of $n$ into distinct parts $\pm 1$ modulo $3$ is equal to the number of partitions of $n$, where the gap between parts is at least $3$ with the gap at least 6 if the parts are multiples of 3.
\end{theorem}

This classical example of congruence--gap partition theorem is well studied and there are many proofs \cite{AG, Alladi_Berkovich, Andrews_Schur1, Andrews_Schur2,Andrews_Capparelli,Bessenrodt, Bressoud}. Out of this long list of proofs, the first and the only polynomial identity that imply Theorem~\ref{Schur_THM} should be credited to Alladi--Berkovich \cite{Alladi_Berkovich}. 

Here we prove a new polynomial identity in the spirit of the polynomial identities that yield Capparelli's partition theorems \cite{BerkovichUncu7}. We will show that the following new polynomial identity implies Schur's partition theorem:

\begin{theorem}\label{Schur_poly_THM} For any fixed integer $N$, let $\N:= \N(m,n_1,n_2):= N-m-n_1-n_2$, then we have
\begin{align}
\label{Schur_poly_EQN}&\sum_{m,n_1,n_2\geq 0} q^{A(n_1,n_2,m)} {3\N\brack m}_q {\N +\left\lfloor \frac{n_1}{2}\right\rfloor \brack \left\lfloor \frac{n_1}{2}\right\rfloor}_{q^6} {\N +\left\lfloor \frac{n_2}{2}\right\rfloor \brack \left\lfloor \frac{n_2}{2}\right\rfloor}_{q^6}
= \sum_{j=-N}^N q^{\frac{j(3j-1)}{2}} {N;j;q^3\choose j}_2, 
\intertext{where}
\label{A_weight}&A(n_1,n_2,m):= \frac{(2m+n_1+n_2 +1)(2m+n_1+n_2 )}{2} + m(n_1+n_2)+ (n_1+n_2)^2 -n_1.
\end{align}
\end{theorem}

The rest of this paper is organized as follows. We start with the necessary definitions that appear in Theorem~\ref{Schur_poly_THM} and the rest of the paper in Section~\ref{Sec_Defs}. A direct proof of Theorem~\ref{Schur_poly_THM} is given in Section~\ref{Sec_Rec_Pf}. Section~\ref{Sec_Combin} has the combinatorial connection of Theorem~\ref{Schur_poly_THM} to Theorem~\ref{Schur_THM} showing that Theorem~\ref{Schur_poly_THM} implies Schur's Theorem. Some $q$-series and combinatorial identities of this study is discussed in Section~\ref{Sec_Tri}.

\section{Necessary Definitions and Some Useful Formulae}\label{Sec_Defs}

In this work, we will use the standard notations \cite{Theory_of_Partitions,Gasper_Rahman,Warnaar}. For variables $a$ and $q$ with $|q|<1$, we define the $q$-Pochhammer symbols and a useful abbreviation as: \begin{align}
\nonumber (a)_\infty:=(a;q)_\infty &:= (1-a)(1-aq)(1-aq^2)(1-aq^3)\dots,\\
\nonumber (a)_n:=(a;q)_n &:= \frac{(a)_\infty }{(aq^n;q)_\infty},\\
\nonumber (a_1,a_2,\dots,a_k;q)_n&:=(a_1)_n(a_2)_n\dots(a_k)_n.\\
\intertext{We note two well known properties of $q$-Pochhammer symbols:}
 \label{Factorial_difference} (a;q)_{n-k} &= \frac{(a;q)_n}{(\frac{q^{1-n}}{a};q)_k} \left(-\frac{q}{a}\right)^{k} q^{{k\choose 2}-nk},
 \intertext{and}
 \label{Factorial_1_over_q} (q^{-1};q^{-1})_n &= (-1)^n q^{-{n+1\choose 2}} (q;q)_n.
\intertext{Let $m,\ n,\ a,$ and $b\in\Z$, we define the $q$-binomial coefficients and the two types of $q$-trinomial coefficients as}
\label{qBinCoeff}{n+m\brack m}_q &:= \left\{ \begin{array}{cl}
\displaystyle\frac{(q)_{n+m}}{(q)_n(q)_m}, &\text{if }n\geq m\geq0,\\ [-1.5ex]\\
0, & \text{otherwise},
\end{array}\right.\\
\label{RoundqTrinomial} {m;b;q\choose a}_2 &:= \sum_{k\geq 0} q^{k(k+b)} {m\brack k}_q {m-k\brack k+a}_q=\sum_{k\geq0} \frac{q^{k(k+b)} (q)_m}{(q)_k(q)_{k+a}(q)_{m-2k-a}},\\
\intertext{and}
\label{TnqTrinomial} T_n{m;q\choose a}&:= q^{\frac{m(m-n)-(a(a-n)}{2}}{m;a-n;q^{-1}\choose a}_2,
\intertext{respectively. The following properties of $q$-binomial coefficients are well known:}
\label{Binom_limit}
 \displaystyle \lim_{n\rightarrow\infty}{n\brack m}_q &= \frac{1}{(q)_m},\\
\intertext{and}
\label{Binom_1_over_q}{n+m\brack m}_{q^{-1}} &= q^{-mn} {n+m\brack m}_{q}.
 \end{align}
 
 \section{Proof of Theorem~\ref{Schur_poly_THM}}\label{Sec_Rec_Pf}
 
 We start by noting that the right-hand side of \eqref{Schur_poly_EQN} is the $S_{3N-1}(1,q):=\R_N(q)$ function defined in \cite{Andrews_Capparelli} that Andrews originally used to prove Schur's theorem directly. In his proof, he shows that this object satisfies the recurrence relation
 \begin{align}
\label{Andrews_recurrence} \R_{N}(q)&= (1+q^{3N-2}+q^{3N-1})\R_{N-1}(q) +q^{3N-3}(1-q^{3N-3})\R_{N-2}(q). 
\intertext{We would like to note that this recurrence can directly be found and automatically proven using Symbolic Computation tools Sigma and qMultiSum \cite{Sigma, qMultiSum}. Same is true for the left-hand side sum of \eqref{Schur_poly_EQN}. Using the mentioned implementations, we first prove that the left-hand side summand, to be denoted by $\F_{N}(m,n_1,n_2)$, satisfies the recurrence}
\nonumber \F_N(m,n_1,n_2) &= \F_{N-1}(m,n_1,n_2) + q^{6N-5}F_{N-2}(m,n_1,n_2-1)+q^{6N-7}F_{N-2}(m,n_1-1,n_2)\\[-1.5ex]\nonumber\\
\label{Left_summand_recurrence}  &+ q^{3N-3}(1+q+q^2)\F_{N-2}(m-1,n_1,n_2) + q^{6N-8}(1+q+q^2)F_{N-3}(m-2,n_1,n_2)\\[-1.5ex]\nonumber\\
\nonumber &- q^{12N-24}\F_{N-4}(m,n_1-1,n_2-1) + q^{9N-15}F_{N-4}(m-3,n_1,n_2),
\intertext{then, by summing \eqref{Left_summand_recurrence} over $m,\ n_1,$ and $n_2$ from 0 to $\infty$, we see that the left-hand side sum $\L_N(q)$ satisfies the recurrence}
\nonumber\L_N(q) &= \L_{N-1}(q) + q^{3N-3}(1+q+q^2+q^{3N-4}+q^{3N-2})\L_{N-2}(q)\\[-1.5ex]\nonumber\\
\label{Left_Sum_recurrence}&\hspace{1cm}+q^{6N-8}(1+q+q^2)\L_{N-3}(q) + q^{9N-15}(1-q^{3N-9})\L_{N-4}(q).
\intertext{Using the recurrence \eqref{Andrews_recurrence} in an iterative fashion on its own terms \[(q^{3N-2}+q^{3N-1})\R_{N-1}(q)\hspace{.8cm}\text{and}\hspace{.8cm}q^{6N-6}\R_{N-2}(q)\] is enough to show that $\R_N(q)$ also satisfies \eqref{Left_Sum_recurrence}. Now that we established that both sides of \eqref{Schur_poly_EQN} satisfy the same recurrence \eqref{Left_Sum_recurrence}, the last task is to check confirm that the first four initial conditions of both sides are the same. For that we give the list:}
\nonumber\L_0(q)=\R_0(q) &= 1,\\
\nonumber\L_1(q)=\R_1(q) &= 1 + q + q^2,\\
\nonumber\L_2(q)=\R_2(q) &= 1 + q + q^2 + q^3 + q^4 + 2 q^5 + q^6 + q^7,\\
\nonumber\L_3(q)=\R_3(q)\\ 
\nonumber &\hspace{-1.3cm}= 1 + q + q^2 + q^3 + q^4 + 2 q^5 + 2 q^6 + 3 q^7+ 3 q^8 + 2 q^9 + 2 q^{10} + 2 q^{11} + 2 q^{12} + 2 q^{13} + q^{14} + q^{15}.
\end{align}
This proves the identity \eqref{Schur_poly_EQN} for any non-negative $N$. For negative values of $N$ both sides of \eqref{Schur_poly_EQN} is 0.

\section{Combinatorics of Theorem~\ref{Schur_poly_THM}}\label{Sec_Combin}

Let $n_1,\ n_2,$ and $m$ be non-negative integers and let the partition $\pi_{n_1,n_2,m}$, to be called \textit{minimal configuration}, be defined as $n_1$ consecutive 1 modulo 3 parts followed by $n_2$ consecutive 2 modulo 3 parts followed by $m$ parts that are exactly 4 apart from their neighboring parts. For positive $n_1,\ n_2,$ and $m$, we have \begin{align}\nonumber\pi_{n_1,n_2,m}:= (&\underline{1,4,7,\dots,3(n_1-1)+1},\ \underline{3n_1+2,3n_1+5,\dots, 3(n_1+n_2-1)+2},\\[-1.5ex]\nonumber\\\label{explicit_minimal_pttn} &\hspace{1cm}3(n_1+n_2)+3,3(n_1+n_2)+7,3(n_1+n_2)+11,\dots,3(n_1+n_2)+3+4(m-1)),\end{align} where we underline the \textit{initial chain} of the $n_1$ consecutive 1 mod 3 parts and also underline the following $n_2$ consecutive 2 mod 3 parts. We do not underline the $m$ 4-apart parts and call these parts \textit{singletons}. If $n_1,\ n_2,$ or $m$ is 0, in the \eqref{explicit_minimal_pttn} we ignore the related portion of the partition with these numbers. As an example, when $n_1=n_2=m=0$, we get an empty list (the unique partition of 0) as our minimal configuration. 

It is easy to see that the minimal configuration $\pi_{n_1,n_2,m}$ satisfies the gap conditions of the Schur Theorem (Theorem~\ref{Schur_THM}). Moreover, this partition has $n_1+n_2+m$ parts and its size is exactly $A(n_1,n_2,m)$ as in \eqref{A_weight}. The name minimal configuration comes from the fact that $\pi_{n_1,n_2,m}$ is the partition with the smallest size that satisfies the gap conditions of Theorem~\ref{Schur_THM} that has $n_1+n_2-2$ gaps of size exactly 3 into $n_1+n_2+m$ parts. 

We would like to start with such a minimal configurations and build up all partitions that satisfy Schur's gap conditions, bijectively. For that we will define ``the forwards motions of the parts'' of the minimal configurations first. This will be done in a similar fashion to \cite{Kagan1,Kagan2,Kagan3,BerkovichUncu7, Uncu2}, mostly resembling the lines of \cite{Uncu2}. 

Before presenting the details, we would like to summarize the way we will approach the forwards motions. First, we will move the singletons; starting from the largest singleton (greatest as an integer) to the smallest singleton. We will preserve the order of the singletons of $\pi_{n_1,n_2,m}$ by moving each part less than or equal to the amount of movement of the previous (greater) part. Then, we will define the motion of the 2 modulo 3 parts as pairs splitting from the end of the 2 modulo 3 initial chain of $\pi_{n_1,n_2,m}$. Once again, this motion will be done starting from the greatest pair (the order with respect to the sum of the pair's parts) to the smallest pair. We will maintain the ordering of the pairs by letting any pair to move at most the same amount as the previous pair that moved before it. We will define crossing over a singleton for these 2 mod 3 pairs, as these pairs may come close to a singleton that moved before any one of the pairs and may violate the Schur's gap conditions. Finally, we will define the motion of the 1 modulo 3 pairs in a similar fashion to the 2 modulo 3 pairs. In this case, we will need the additional treatment of a 1 modulo 3 pair crossing over consecutive 2 modulo 3 parts of the partition. All the defined motions will bijective maps and at each step we will make sure the outcome partition satisfies the Theorem~\ref{Schur_THM}'s gap conditions.

Starting from the largest part (the last part) we can move the $m$-singletons forwards by adding each element a non-negative value: $r_m$ to the largest part, $r_{m-1}$ to the second largest with $r_m\geq r_{m-1}$... $r_1$ to the smallest singleton $r_2\geq r_1\geq 0$. The order $0 \leq r_1\leq r_2\leq \dots \leq r_m$ is enough to ensure that order of the singletons are preserved after the motions. Such a list $(r_1,r_2,\dots,r_m)$ with $0 \leq r_1\leq r_2\leq \dots \leq r_m$ may not be a partition itself; some $r_i$ values might be 0. On the other hand, by ignoring the zero values, it is clear that every such list (used in the forwards motion of the singletons) corresponds to a unique partition into $\leq m$ parts. Therefore, the generating function that is related with the forwards motions of $m$ singletons is the generating function for the number of partitions into $\leq m$ parts:  \begin{equation}\label{singleton_move_GF}\frac{1}{(q)_m}.\end{equation} It is clear that the motions of the singletons are bijective and can easily be reversed. 

After moving the singletons, we start moving the initial chain of the $n_2$ 2 modulo 3 parts (signified by the underlining of all the related parts). In this motion we first split the last two elements of the initial chain, making them a pair (signified by under-braces) \begin{align*}&\underline{3n_1+2,3n_1+5,\dots, 3(n_1+n_2-2)+2,3(n_1+n_2-1)+2}\\[-1.5ex]\\&\hspace{1cm}\mapsto \underline{3n_1+2,3n_1+5,\dots, 3(n_1+n_2-3)+2}\underbrace{3(n_1+n_2-2)+2,3(n_1+n_2-1)+2}.\end{align*} Later we will start moving these pairs by moving one to the next possible location where the numbers again become a pair of consecutive 2 modulo 3 parts. Before doing so, note that we are splitting and moving two parts of an $n_2$ length initial chain together. Hence, we can at most split and move $\lfloor n_2/2 \rfloor$ pairs. In the motion of these pairs, similar to the singletons case, we will move the greatest pair (ordered with respect to sum of the parts in the pair) forwards the most, then the second largest pair less than the motion of the first pair etc. 

For a given pair $\underbrace{x,y}$ of $\pi$ that satisfies the gap conditions of Schur's theorem (Theorem~\ref{Schur_THM}), if $\pi$ does not have a part $z$ such that $y+3 \leq z < y+6$, we define the motion of this pair as 
\begin{equation}\label{pair_free_motion}                                                                                                                                              
\underbrace{x,y}\mapsto \underbrace{x+3,y+3}.
\end{equation} This forwards motion adds a total of 6 to the size of the partition $\pi$, the greater part of the pair moves 3 steps forwards, and it does not change the residue class of $x$ and $y$ modulo 3. Moreover, it is clearly bijective and can be undone. 

There might be a $z$ value that is in 4 or 5 distance to the larger part of the pair that we would like to move. This forwards motions needs us to define particular bijective rules so that the outcome partition would still satisfy the gap conditions of Schur's theorem. Given a pair $\underbrace{3k+2, 3k+5}$, we define the following bijective rules for crossing singletons. Similar to adjustments explained in \cite{Kagan3}, we need to handle different cases differently. These cases will depend on the number of singletons that one pair needs to cross in a given circumstance:
\begin{itemize}
\item{Case 1: } If the pair is crossing a single singleton (that is $\leq 6$ distant to the pair and it is more than 6 distant to the following larger part (if any), we define the following bijective forwards motion. For $r = 0,1,2$, we have
\begin{equation}
 \label{jump_over_singletons}\underbrace{3k+2, 3k+5},3k+8+r  \mapsto 3k+2+r, \underbrace{3k+8, 3k+11}.
 \end{equation}
\item{Case 2: } If the pair is crossing two close singletons (a singleton followed by another singleton that is $\leq5$ distant) where employing a case of the \eqref{jump_over_singletons} would break the Schur's gap conditions, we use the following bijective motions. For $r,s\in \{0,1,2\}$ with $s-r\leq 1$, we define the motions:
\begin{equation}
\label{jump_over_doubletons} \underbrace{3k+2, 3k+5},3k+8+r,3k+12+s  \mapsto 3k+2+r,3k+6+s, \underbrace{3k+11, 3k+14}.
\end{equation}
Notice here that the $(r,s)=(2,3)$ possibility is excluded although this can be considered as two close singletons. That is because in this case we can use \eqref{jump_over_singletons} with $r=2$ and this would not break the Schur's gap conditions.
\item{Case 3:} If the pair to move needs to cross three close singletons, and if employing the motions \eqref{jump_over_singletons} or \eqref{jump_over_doubletons} is violating the gap conditions of Theorem~\ref{Schur_THM}. Let $r,s,t\in\{0,1\}$ with $s-r\leq 1$ and $t-s\leq 1$ then we have define the bijective motion:
\begin{equation}
\label{jump_over_tripletons} \underbrace{3k+2, 3k+5},3k+8+r,3k+12+s,3k+16+t  \mapsto 3k+2+r,3k+6+s,3k+10+t, \underbrace{3k+14, 3k+17}.
\end{equation}
Observe that a possible part of the partition (if any) that follows the part $3k+16+t$ in \eqref{jump_over_tripletons} is at least of size $3k+20+t$. The gap between the largest part of our last motion \eqref{jump_over_tripletons} $3k+17$ has at least a gap of 3 with this possible part $3k+20+t$. Therefore, one can stop the crossing of the pairs over singletons here. This also means one can stop defining particular rules here as well. If they would like to move the pair $\underbrace{3k+14, 3k+17}$ once again, they can start with checking and employing the bijective motion rules \eqref{pair_free_motion}-\eqref{jump_over_tripletons}.
\end{itemize}
Hence, the list of motions \eqref{pair_free_motion}, \eqref{jump_over_singletons}, \eqref{jump_over_doubletons}, and \eqref{jump_over_tripletons} is the full bijective list of motions for the $\lfloor n2/2 \rfloor$ 2 modulo 3 pairs. Furthermore, each of these motions add 6 to the total size of the partition once employed. Recalling that a pair can move at most the same amount as the previous pair is enough to see that the generating function related with the motions of the $\lfloor n2/2 \rfloor$ 2 modulo 3 pairs is in bijection with the partitions into $\leq \lfloor n2/2 \rfloor$ parts. The generating function for the forwards motions of the 2 modulo 3 initial chain is \begin{equation} \label{n2_move_GF}
\frac{1}{(q^6;q^6)_{\lfloor n2/2 \rfloor}}.
\end{equation}
Also, observe that in all these motions the pairs move \begin{equation}\label{Num_of_forwards_motions_in_crossing_singletons}3+ 3\times\text{``the number of singletons crossed"}\end{equation} steps forwards.

Finally, we move on to the motions starting from the initial chain of the $n_1$ 1 modulo 3 parts. Similar to the previous case, we first split the last two elements of the initial chain, making them a pair (signified by under-braces) \begin{align*}&\underline{1,4,7,\dots,3(n_1-3)+1,3(n_1-2)+1,3(n_1-1)+1}\\[-1.5ex]\\&\hspace{1.5cm}\mapsto \underline{1,4,7,\dots,3(n_1-3)+1},\underbrace{3(n_1-2)+1,3(n_1-1)+1}.\end{align*} Similar to the previous (2 modulo 3 initial chain) case we can split and move at most $\lfloor n_1/2\rfloor$. Moreover, \eqref{pair_free_motion} is still valid for this case, and for the rest of the crossing rules all one needs to do is to use the same cases related to \eqref{jump_over_singletons}-\eqref{jump_over_tripletons} and subtract 1 from each and every term in these motions. All the size and number of forward motion observations that is made for the 2 modulo 3 pairs are still valid for the 1 modulo 3 pairs.

One new situation in this case appears if a 1 modulo 3 pair comes close to a group of consecutive 2 modulo 3 parts of the partition. In this situation, we define the following bijective map. Let $l\geq 3$ be the number of consecutive 2 modulo 3 parts, then
\begin{align}
\nonumber \underbrace{3k+1,3k+4}&,3k+8, 3k+11,\dots, 3k+3l+5\\
\label{jump_over_2mod3_chain}&\mapsto 3k+2, 3k+5, \dots, 3k+3l-1, \underbrace{3k+3l+4,3k+3l+7}.
\end{align}
Note that $l=1$ and $2$ cases are covered under the relative versions of \eqref{jump_over_singletons} and \eqref{jump_over_doubletons} for the 1 modulo 3 pairs. Moreover, note that in this forwards motion the pair makes $l$ extra motions and again the size of the overall partition raises only by 6. By the same argument as the previous case now we can see that the generating function corresponding to the forwards motion of the 1 modulo 3 initial chain is \begin{equation}
\label{n1_move_GF} \frac{1}{(q^6;q^6)_{\lfloor n_1/2\rfloor}}.
\end{equation}

Combining \eqref{singleton_move_GF}, \eqref{n2_move_GF} and \eqref{n1_move_GF}, it is easy to see that \begin{equation}
\frac{q^{A(n_1,n_2,m)}}{(q^6;q^6)_{\lfloor n_1/2\rfloor}(q^6;q^6)_{\lfloor n_2/2\rfloor}(q)_m}
\end{equation} is the generating function for all the partitions that satisfies the gap conditions of Theorem~\ref{Schur_THM} that can be constructed from the minimal configuration $\pi_{n_1,n_2,m}$ defined in \eqref{explicit_minimal_pttn}, where $A(n_1,n_2,m)$ is as defined in \eqref{A_weight}. By summing over all possible $n_1, n_2,$ and $m$ we get the following theorem.

\begin{theorem}\label{Ali_Schur_GF_THM} Let $A(n_1,n_2,m)$ be as defined in \eqref{A_weight}, then
\begin{equation}\label{Ali_Schur_GF}
\sum_{n_1,n_2,m\geq 0} \frac{x^{n_1+n_2+m}q^{A(n_1,n_2,m)}}{(q^6;q^6)_{\lfloor n_1/2\rfloor}(q^6;q^6)_{\lfloor n_2/2\rfloor}(q)_m}
\end{equation}
is the generating function for the number of partitions that satisfy the gap conditions of Schur's theorem (Theorem~\ref{Schur_THM}), where the exponent of $x$ counts the number of parts of the counted partitions.
\end{theorem}

The triple series \eqref{Ali_Schur_GF} is the analogue of the double sums presented for the G\"ollnitz--Gordon and little G\"ollnitz theorems in \cite{Uncu2}. This series (as well as the ones in \cite{Uncu2}) are inspired by Kur\c{s}ung\"oz's recent works \cite{Kagan1,Kagan2,Kagan3}.  Due to the difference in the minimal configuration setups and some of the motions, the author and Kur\c{s}ung\"oz gets equivalent but different representations for the same generating functions. Here we present Kur\c{s}ung\"oz's version of the generating function represented in Theorem~\ref{Ali_Schur_GF_THM}.

\begin{theorem}[Kur\c{s}ung\"oz, 2018]\label{Kagan_Schur_GF_THM} Let \begin{equation}\label{K_weight}K(n_1,n_2,m):= 6(n_1+n_2)^2+2m^2+6m(n_1+n_2)-n_1+n_2-m,\end{equation} then
\begin{equation}\label{Kagan_Schur_GF}
\sum_{n_1,n_2,m\geq 0} \frac{x^{2n_1+2n_2+m}q^{K(n_1,n_2,m)}}{(q^6;q^6)_{n_1}(q^6;q^6)_{n_2}(q)_m}
\end{equation}
is the generating function for the number of partitions that satisfy the gap conditions of Schur's theorem (Theorem~\ref{Schur_THM}), where the exponent of $x$ counts the number of parts of the counted partitions.
\end{theorem}

To avoid any speculative trivial transformation between \eqref{Ali_Schur_GF} and \eqref{Kagan_Schur_GF} please note that \begin{equation}\label{exponent_difference}A(2n_1,2n_2,m)-K(n_1,n_2,m) = 2m.\end{equation} We would also like to present the equality of the series \eqref{Ali_Schur_GF} and \eqref{Kagan_Schur_GF} after doing even-odd splits for the variables $n_1$ and $n_2$ and regrouping in \eqref{Ali_Schur_GF}. We will also be using \eqref{exponent_difference} to write the $q$-factors in the summands using the same quadratic $K(n_1,n_2,m)$. 

\begin{theorem} We have
\begin{align}
\nonumber\sum_{n_1,n_2,m\geq 0}& \frac{x^{2n_1+2n_2+m}q^{K(n_1,n_2,m)+2m}}{(q^6;q^6)_{n_1}(q^6;q^6)_{n_2}(q)_m} (1+xq^{6n_1+6n_2+3m+1}+xq^{6n_1+6n_2+3m+2}+x^2q^{12n_1+12n_2+6m+6})\\&=
\sum_{n_1,n_2,m\geq 0} \frac{x^{2n_1+2n_2+m}q^{K(n_1,n_2,m)}}{(q^6;q^6)_{n_1}(q^6;q^6)_{n_2}(q)_m},
\end{align} where $K(n_1,n_2,m)$ is as in \eqref{K_weight}.
\end{theorem}

Now we start finding a finite analogue of \eqref{Ali_Schur_GF}. Let $N$ be a non-negative integer. We would like to find all the partitions with the largest part $\leq N$ that are counted by \eqref{Ali_Schur_GF}. For that we need to count how many times a singleton, a 2 modulo 3 pair and a 1 modulo 3 pair can move forward before exceeding $N$ and change our generating functions from reciprocal of a $q$-factorials to the necessary $q$-binomials. 

The largest singleton of the minimal configuration $\pi_{n_1,n_2,m}$, $3(n_1+n_2-1)+2$, can only move $N-[3(n_1+n_2)+3+4(m-1)]$ steps forward before exceeding the imposed bound. Therefore, with the new bound, the motions for the singletons is related with the partitions into $\leq m$ parts, where each part is $\leq N-[3(n_1+n_2-1)+2]$. The generating function for all such partitions is \begin{equation}
\label{singleton_move_GF_BDD} {N-[3(n_1+n_2)+3+4(m-1)] + m \brack m}_q 
\end{equation}
Each forwards movement of a 2 modulo 3 pair gets it 3 units closer to the bound $N$. Then, ignoring the singletons for a second, the largest 2 modulo 3 pair $\underbrace{3(n_1+n_2-2)+2,3(n_1+n_2-1)+2}$ can move at most \[\left\lfloor \frac{N-[3(n_1+n_2-1)+2]}{3} \right\rfloor\] steps forwards before the larger part, $3(n_1+n_2-1)+2$, of the pair goes over the bound on the largest part $N$. Recall \eqref{Num_of_forwards_motions_in_crossing_singletons}: crossing over singletons make these pairs move extra steps forwards. There are $m$ singletons that are greater than the largest pair $\underbrace{3(n_1+n_2-2)+2,3(n_1+n_2-1)+2}$. Hence before reaching the bound this pair would need to cross all of those $m$ singletons, and move an extra 3 steps forwards each time. Therefore, the actual number of steps this pair can take forwards before passing the bound $N$ is \[\left\lfloor \frac{N-[3(n_1+n_2-1)+2]}{3} \right\rfloor - m.\] This shows us that the bounded forwards motion of the 2 modulo 3 pairs is related with partitions into $\leq \lfloor n_2/2\rfloor$ parts each $\leq \lfloor {N-[3(n_1+n_2-1)+2]/3} \rfloor - m$. This implies that the related generating function for this motion (that changes the size by 6 each time) is
\begin{equation}\label{n2_move_GF_BDD}\left[\hspace{-2mm}\begin{array}{c}\left\lfloor \frac{N-[3(n_1+n_2-1)+2]}{3} \right\rfloor - m + \left\lfloor \frac{n_2}{2} \right\rfloor \\ \left\lfloor \frac{n_2}{2}\right\rfloor \end{array}\hspace{-2mm}\right]_{q^6}.
\end{equation}
Finally, Similar to the previous case, forgetting about the the $n_2$ 2 modulo 3 parts and the $m$ singletons, the largest 1 modulo 3 pair, $\underbrace{3(n_1-2)+1,3(n_1-1)+1}$, can move \[\left\lfloor \frac{N-[3(n_1-1)+1]}{3} \right\rfloor\] forwards before $3(n_1-1)+1$ goes over $N$. Including our observations about the extra steps one pair takes while crossing over parts, we see that the actual number of steps forwards that the largest pair can take is \[\left\lfloor \frac{N-[3(n_1-1)+1]}{3} \right\rfloor-m-n_2.\] With that, similar to the previous case, we see that the generating function related to the forwards motions of the $\lfloor n_1/2 \rfloor$ 1 modulo 3 pairs is \begin{equation}\label{n1_move_GF_BDD}
\left[\hspace{-2mm}\begin{array}{c}\left\lfloor \frac{N-[3(n_1-1)+1]}{3} \right\rfloor - m - n_2 + \left\lfloor \frac{n_1}{2} \right\rfloor \\ \left\lfloor \frac{n_1}{2}\right\rfloor \end{array}\hspace{-2mm}\right]_{q^6}.
\end{equation}

Putting \eqref{singleton_move_GF_BDD}, \eqref{n2_move_GF_BDD}, and \eqref{n1_move_GF_BDD} together, we get that \begin{align}
\label{Ali_Schur_BDD_Summand}q^{A(n_1,n_2,m)}&{N-3(n_1+n_2 + m)+1 \brack m}_q \\\nonumber &\times\left[\hspace{-2mm}\begin{array}{c}\left\lfloor \frac{N-[3(n_1-1)+1]}{3} \right\rfloor - m - n_2 + \left\lfloor \frac{n_1}{2} \right\rfloor \\ \left\lfloor \frac{n_1}{2}\right\rfloor \end{array}\hspace{-2mm}\right]_{q^6} \left[\hspace{-2mm}\begin{array}{c}\left\lfloor \frac{N-[3(n_1+n_2-1)+2]}{3} \right\rfloor - m + \left\lfloor \frac{n_2}{2} \right\rfloor \\ \left\lfloor \frac{n_2}{2}\right\rfloor \end{array}\hspace{-2mm}\right]_{q^6} 
\end{align} is the generating function for the number of partitions that satisfies the gap conditions of Theorem~\ref{Schur_THM} that can be constructed from the minimal configuration $\pi_{n_1,n_2,m}$ with the extra bound on the largest part $\leq N$. Summing \eqref{Ali_Schur_BDD_Summand} over $n_1,\ n_2,$ and $m$ yields the following theorem.

\begin{theorem}\label{Ali_Schur_GF_BDD_THM} For any non-negative integer $N$, the expression 
\begin{align}\label{Ali_Schur_GF_BDD}
\sum_{n_1,n_2,m\geq0 }x^{n_1+n_2+m} &q^{A(n_1,n_2,m)}{N-3(n_1+n_2 + m)+1 \brack m}_q \\\nonumber &\times\left[\hspace{-2mm}\begin{array}{c}\left\lfloor \frac{N-[3(n_1-1)+1]}{3} \right\rfloor - m - n_2 + \left\lfloor \frac{n_1}{2} \right\rfloor \\ \left\lfloor \frac{n_1}{2}\right\rfloor \end{array}\hspace{-2mm}\right]_{q^6} \left[\hspace{-2mm}\begin{array}{c}\left\lfloor \frac{N-[3(n_1+n_2-1)+2]}{3} \right\rfloor - m + \left\lfloor \frac{n_2}{2} \right\rfloor \\ \left\lfloor \frac{n_2}{2}\right\rfloor \end{array}\hspace{-2mm}\right]_{q^6}
\end{align} where $A(n_1,n_2,m)$ is defined as in Theorem~\ref{Schur_poly_THM}, is the generating function for the number of partitions that satisfy the gap conditions of Theorem~\ref{Schur_THM} with the extra condition that each part is $\leq N$, where the exponent of x counts the number of parts.
\end{theorem}

One direct corollary of Theorem~\ref{Ali_Schur_GF_BDD_THM} is the interpretation of the left-hand side of \eqref{Schur_poly_EQN} when $N\mapsto 3N-1$. 

\begin{corollary}\label{Cor1}For any positive integer $N$, and $\N:= N-n_1-n_2-m$, the expression \[\sum_{m,n_1,n_2\geq 0} q^{A(n_1,n_2,m)} {3\N\brack m}_q {\N +\left\lfloor \frac{n_1}{2}\right\rfloor \brack \left\lfloor \frac{n_1}{2}\right\rfloor}_{q^6} {\N +\left\lfloor \frac{n_2}{2}\right\rfloor \brack \left\lfloor \frac{n_2}{2}\right\rfloor}_{q^6},\]where $A(n_1,n_2,m)$ is defined as in Theorem~\ref{Schur_poly_THM}, is the generating function for the number of partitions that satisfy the gap conditions of Theorem~\ref{Schur_THM} with the extra condition that each part is $\leq 3N-1$.
\end{corollary}

On the other hand, Andrews \cite{Andrews_Capparelli} interpreted the right-hand side of \eqref{Schur_poly_EQN} as the same generating function in the interpretation of Corollary~\ref{Cor1}. This is also proves the validity of Theorem~\ref{Ali_Schur_GF_THM} for positive values of $N$, this time using only the combinatorial constructions. In \cite[(3.9), pg. 147]{Andrews_Capparelli}, Andrews also shows that the right-hand side sum converges to the generating function for the number of partitions into distinct parts $\pm 1$ mod 3: \[(-q,-q^2;q^3)_\infty.\] This shows that after taking limits $N\rightarrow\infty$ of \eqref{Schur_poly_EQN}, and using \eqref{Binom_limit} as needed, we have \[\sum_{n_1,n_2,m\geq 0} \frac{x^{n_1+n_2+m}q^{A(n_1,n_2,m)}}{(q^6;q^6)_{\lfloor n_1/2\rfloor}(q^6;q^6)_{\lfloor n_2/2\rfloor}(q)_m} = (-q,-q^2;q^3)_\infty,\] which is the analytic version of the Schur's theorem (Theorem~\ref{Schur_THM}). This shows that the polynomial identity \eqref{Schur_poly_EQN} (keeping the interpretation, Theorem~\ref{Ali_Schur_GF_BDD_THM} in mind) implies Theorem~\ref{Schur_THM}.

\section{Some Implications of Theorem~\ref{Schur_poly_THM}}\label{Sec_Tri}

We start by sending $q\mapsto 1/q$ in \eqref{Schur_poly_EQN} followed by the use of \eqref{Binom_1_over_q} and multiplying both sides with $q^{3N^2/2}$. This yields the equivalent formula
\begin{equation}
\label{Schur_poly_Dual_EQN}\sum_{m,n_1,n_2\geq 0} q^{B(n_1,n_2,m,N)-A(n_1,n_2,m)} {3\N\brack m}_q {\N +\left\lfloor \frac{n_1}{2}\right\rfloor \brack \left\lfloor \frac{n_1}{2}\right\rfloor}_{q^6} {\N +\left\lfloor \frac{n_2}{2}\right\rfloor \brack \left\lfloor \frac{n_2}{2}\right\rfloor}_{q^6}
= \sum_{j=-N}^N q^{\frac{j}{2}} T_0{N;q^3\choose j},
\end{equation}
where $A(n_1,n_2,m)$ is as in \eqref{A_weight}, $\N=N-n_1=n_2=m$, and
\begin{equation}
\label{B_weight} B(n_1,n_2,m,N) = \frac{3N^2}{2} - (3\N-m)m - 6\N\left( \left\lfloor \frac{n_1}{2}\right\rfloor+ \left\lfloor \frac{n_2}{2}\right\rfloor\right).
\end{equation}

Note that the sides in \eqref{Schur_poly_Dual_EQN} are not polynomials but multiplying both sides with $q^{N/2}$ is enough to make them polynomials. After multiplying both sides of \eqref{Schur_poly_Dual_EQN} by $q^{N/2}$, writing the definition of \eqref{RoundqTrinomial} in for the right-hand side of \eqref{Schur_poly_Dual_EQN} and using \eqref{Factorial_1_over_q} multiple times we see that \begin{equation}
\label{T0_open} \sum_{j=-\infty}^{\infty} q^{\frac{N+j}{2}} T_0{N;q^3\choose j} = \sum_{k,l\geq 0} \frac{q^{k+\frac{l(3l+1)}{2}} (q^3;q^3)_N}{(q^3;q^3)_{N-k-l}(q^3;q^3)_k(q^3;q^3)_l}, 
\end{equation} after simple changes of variables. We use \eqref{Factorial_difference} for the term $(q^3;q^3)_{(N-k)-l}$ to separate the difference of the variable $l$. This way we end up with the expression \[\sum_{k\geq 0} q^k {N\brack k}_{q^3} \sum_{l\geq 0}\frac{(q^{-3(N-k)};q^3)_l}{(q^3;q^3)_l}\left(-q^{3(N-k)+2}\right)^l.\] The inner sum can be summed using the $q$-binomial theorem \cite[II.4, p 354]{Gasper_Rahman}, and we get \begin{equation}
\label{T_0_to_binomials} \sum_{j=-\infty}^{\infty} q^{\frac{N+j}{2}} T_0{N;q^3\choose j} = \sum_{k\geq0} q^k{N\brack k}_{q^3} (-q^2;q^3)_{N-k}.
\end{equation} Not only that, \eqref{T_0_to_binomials} with the use of \cite[II.1, p 354]{Gasper_Rahman} on the right-hand side, yields \begin{equation}\label{T_0_limit}\lim_{N\rightarrow\infty}  \sum_{j=-\infty}^{\infty} q^{\frac{N+j}{2}} T_0{N;q^3\choose j} = \frac{1}{(q^2;q^3)_\infty(q;q^6)_\infty}.
\end{equation}

To evaluate the $N\rightarrow\infty$ limit on the left-hand side of \eqref{Schur_poly_Dual_EQN} with the extra $q^{N/2}$, one first needs to make a change of summation variables and rewrite the $q$-factor. We would like to use $y=\N$ as our summation variable instead of $n_2$, but the parity of $N$ must be kept in check to correctly identify the exponent of the $q$-factor in this case. Let $r(a,b)$ be the remainder of the division $a \div b$, for $a,b\in \mathbb{N}$. After the change of variables, the left-hand side of \eqref{Schur_poly_Dual_EQN} multiplied with an extra $q^{N/2}$ becomes \[\sum_{m,n_1,y\geq 0} q^{Q(m,n_1,y,N)} {3y\brack m}_q {y +\left\lfloor \frac{n_1}{2}\right\rfloor \brack y}_{q^6} \left[\hspace{-2mm}\begin{array}{c}y +\left\lfloor \frac{N-m-n_1-y}{2}\right\rfloor \\ y\end{array}\hspace{-2mm}\right]_{q^6},\] where \[Q(m,n_1,y,N) = {m\choose 2} +\frac{y(3y+1)}{2} +n_1 + 3y\, r(N+m+y,2) + 6y\,r(n_1,2)\, r(N+m+y+1,2).\] Then, by taking the limit $N\rightarrow\infty$ for odd and even $N$ and using \eqref{T_0_limit} we get the following theorem.

\begin{theorem} Let $t=1,\ 2$, then
\begin{equation}
\sum_{m,n_1,y\geq 0} \frac{q^{Q_t(m,n_1,y)}}{(q^6;q^6)_y} {3y\brack m}_q {y +\left\lfloor \frac{n_1}{2}\right\rfloor \brack y}_{q^6}= \frac{1}{(q^2;q^3)_\infty(q;q^6)_\infty},
\end{equation}
where
\begin{equation}
Q_t(m,n_1,y) = {m\choose 2} +\frac{y(3y+1)}{2} +n_1 + 3y\, r(m+y+t,2) + 6y\,r(n_1,2)\, r(m+y+1+t,2)
\end{equation}
\end{theorem}

Recall that Warnaar \cite[(10), pg 2516]{Warnaar} proved the following summation formula.
\begin{equation}
\label{T0_BL} \sum_{i\geq 0} q^{\frac{i^2}{2}} {L\brack i}_q T_0{i;q\choose a} = q^{\frac{a^2}{2}} {2L\brack L-a}_q.
\end{equation}
This can be applied to the right-side of \eqref{Schur_poly_Dual_EQN} to get the following theorem.

\begin{theorem}\label{Summation_Formula_THM}Let $\N=N-m-n_1-n_2$, for any non-negative integer $M$ we have
\begin{align}
\label{Summation_formula}\sum_{N,m,n_1,n_2\geq 0} q^{ \frac{3N^2}{2}+B(n_1,n_2,m,N)-A(n_1,n_2,m)}&{3\N\brack m}_q {M\brack N}_{q^3} {\N +\left\lfloor \frac{n_1}{2}\right\rfloor \brack \left\lfloor \frac{n_1}{2}\right\rfloor}_{q^6} {\N +\left\lfloor \frac{n_2}{2}\right\rfloor \brack \left\lfloor \frac{n_2}{2}\right\rfloor}_{q^6}= (-q,-q^2;q^3)_{M},
\end{align}where $A(n_1,n_2,m)$ and $B(n_1,n_2,m,N)$ are defined as in \eqref{A_weight} and \eqref{B_weight}, respectively.
\end{theorem}

\begin{proof}
We sum both sides of \eqref{Schur_poly_Dual_EQN} over $N$ from 0 to $M$ after multiplying the summand with \[q^{\frac{3N^2}{2}}{M\brack N}_{q^3}.\] This gives the left-hand side of \eqref{Summation_formula}. For the right-hand side of the formula, we interchange the order of summations, use \eqref{T0_BL} followed by the summation formula \cite[(3.3.6). p. 36]{Theory_of_Partitions}. This yields \[q^{\frac{(3M+1)M}{2}} (-q^{1-3M};q^3)_{2M},\] which after basic simplifications is equal to the right-hand side of the equation \eqref{Summation_formula}. 
\end{proof}

The limit $M\rightarrow \infty$ of \eqref{Summation_formula} is much more straightforward than the limit $n\rightarrow\infty$. By employing \eqref{Binom_limit}, we get the following corolary of Theorem~\ref{Summation_Formula_THM}.

\begin{corollary}\[\sum_{N,m,n_1,n_2\geq 0} \frac{q^{ \frac{3N^2}{2}+B(n_1,n_2,m,N)-A(n_1,n_2,m)}}{(q^3;q^3)_N}{3\N\brack m}_q {\N +\left\lfloor \frac{n_1}{2}\right\rfloor \brack \left\lfloor \frac{n_1}{2}\right\rfloor}_{q^6} {\N +\left\lfloor \frac{n_2}{2}\right\rfloor \brack \left\lfloor \frac{n_2}{2}\right\rfloor}_{q^6}= (-q,-q^2;q^3)_{\infty}\]
where $A(n_1,n_2,m)$ and $B(n_1,n_2,m,N)$ are defined as in \eqref{A_weight} and \eqref{B_weight}, respectively.
\end{corollary}

Theorem~\ref{Schur_poly_THM} (and the equation \eqref{Schur_poly_Dual_EQN}) and Theorem~\ref{Summation_Formula_THM} also yield some intriguing combinatorial corollaries at the $q=1$ level.

\begin{corollary} For some non-negative integer $M$, $\N:=N-n_1+n_2-m$ and $\mathcal{M} := M-n_1-n_2-m$, we have
\begin{align}
\label{q=1_1}\sum_{m,n_1,n_2\geq 0} {3\mathcal{M}\choose m} {\mathcal{M} +\left\lfloor \frac{n_1}{2}\right\rfloor \choose \mathcal{M}} {\mathcal{M} +\left\lfloor \frac{n_2}{2}\right\rfloor \choose\mathcal{M}} &= 3^M,
\intertext{and}
\label{q=1_2}\sum_{N,m,n_1,n_2\geq 0} {M\choose N}{3\N\choose m} {\N +\left\lfloor \frac{n_1}{2}\right\rfloor \choose \N} {\N +\left\lfloor \frac{n_2}{2}\right\rfloor \choose \N}&= 4^{M}.
\end{align}
\end{corollary}

\begin{proof}The equation \eqref{q=1_2} is a clear consequence of \eqref{Summation_formula}, or one can get it from \eqref{q=1_1} as it is the classical binomial theorem. For the equation \eqref{q=1_1}, one only needs to recall that \[\sum_{j=-N}^N x^j {N;j;1\choose j}_2 = (x^{-1} + 1+ x)^N,\] and set $x$ to 1.
\end{proof}

\section{Acknowledgments}
The author would like to thank Karl Mahlburg for bringing \cite{Kagan3} to our attention and for his interest. The author would also like to thank Alexander Berkovich for the stimulating discussion and his suggestions on the manuscript.

Research of the author is supported by the Austrian Science Fund FWF, SFB50-07 and SFB50-09 Projects.

\end{document}